\documentclass[12pt]{article}
\usepackage{amsfonts}
\usepackage{amsmath,color,amsthm,amssymb,cite,bbm,mathrsfs}
\usepackage{epsfig}
\usepackage{graphicx}
\usepackage{latexsym}          
\usepackage{tikz}

\usepackage{color}
\title{Contributions to the Theory of de Bruijn Cycles}
\author {Andr\'e Alexander Campbell \& Anant P.~Godbole\\
Department of Mathematics and Statistics\\
East Tennessee State University\and Bill Kay\\
Department  of Mathematics and Computer Science\\
Emory University}
\begin{document}
\def\qed{\vbox{\hrule\hbox{\vrule\kern3pt\vbox{\kern6pt}\kern3pt\vrule}\hrule}}
\def\ms{\medskip}
\def\n{\noindent}
\def\ep{\varepsilon}
\def\G{\Gamma}
\def\lr{\left(}
\def\ls{\left[}
\def\rs{\right]}
\def\lf{\lfloor}
\def\rf{\rfloor}
\def\lg{{\rm lg}}
\def\lc{\left\{}
\def\rc{\right\}}
\def\rr{\right)}
\def\ph{\varphi}
\def\p{\mathbb P}
\def\nk{n \choose k}
\def\ca{{\cal A}}
\def\cb{{\cal B}}
\def\cw{{\cal W}}
\def\s{\cal S}
\def\e{\mathbb E}
\newtheorem{thm}{Theorem}
\newtheorem{con}{Conjecture}
\newtheorem{claim}[thm]{Claim}
\newtheorem{definition}[thm]{Definition}
\newtheorem{lem}[thm]{Lemma}
\newtheorem{cor}[thm]{Corollary}
\newtheorem{remark}[thm]{Remark}
\newtheorem{prp}[thm]{Proposition}
\newtheorem{ex}[thm]{Example}
\newtheorem{eq}[thm]{equation}
\newtheorem{que}{Problem}[section]
\newtheorem{ques}[thm]{Question}
\providecommand{\floor}[1]{\left\lfloor#1\right\rfloor}
\maketitle
\begin{abstract}  A de Bruijn cycle is a cyclic listing of length $\ca$, of a collection of $\ca$ combinatorial objects, so that each object appears exactly once as a set of consecutive elements in the cycle.  In this paper, we show the power of de Bruijn's original theorem, namely that the cycles bearing his name exist for $n$-letter words on a $k$-letter alphabet for all values of $k,n$, to prove that we can create de Bruijn cycles for the assignment of elements of $[n]=\{1,2,\ldots,n\}$ to the sets in any labeled subposet of the Boolean lattice; de Bruijn's theorem corresponds to the case when the subposet in question consists of a single ground element.  The landmark work of \cite{cdg} extended the agenda of finding de Bruijn cycles to possibly the next most natural set of combinatorial objects, namely $k$-subsets of $[n]$.  In this area, important contributions have been those of \cite{h} and \cite{r}.  Here we follow the direction of \cite{bg}, who proved that, in a suitable encoding, de Bruijn cycles can be created for the subsets of $[n]$ of size in the interval $[s,t]; 0\le s<t\le n$.  In this paper we generalize this result to exhibit de Bruijn cycles for words with weight between $s$ and $t$, where these parameters are suitably restricted.
\end{abstract}
\section{Introduction}   A de Bruijn cycle  is a cyclic listing of length $\ca$, of a collection of $\ca$ combinatorial objects, so that each object appears exactly once as a set of consecutive elements in the cycle.  For example, the cyclic list 11101000 encodes each of the eight binary three letter words so that each appears exactly once.  Viewed differently, and using a different coding, we see that the same list encodes all the 8 subsets of $[3]=\{1,2,3\}$, with the convention that, e.g., the string 101 represents the subset $\{1,3\}$ -- whose characteristic vector it happens to be.  de Bruijn's theorem states that this example is not an anomaly, and that de Bruijn cycles exist for the collection of $n$-letter words on the $k$-letter alphabet $\{0,1,\ldots,k-1\}$, and thus, using a different coding, for all the multisets of $[n]=\{1,2,\ldots,n\}$ where each element may appear at most $k-1$ times in a multiset.  {\it A de Bruijn cycle is most often called a Universal Cycle (or U-cycle) in general contexts, but we shall not do so in this paper.}

Consider the poset $\cb=\cb_n$, the so-called Boolean Lattice of all subsets of $[n]$ ordered by inclusion.  A {\it subposet} of $\cb$ is defined in this paper to be a collection of subsets of $\cb$ with the same Hasse diagram that is, moreover, consistent with the Hasse diagram of $\cb$.  For example, the W-poset or a 4-chain are each  consistent with $\cb_3$ but not with $\cb_2$.  In Section 2 of this paper, we show the versatility of de Bruijn's original theorem by showing that we can create de Bruijn cycles for the assignment of elements of $[n]=\{1,2,\ldots,n\}$ to the sets in {\it any labeled} subposet of the Boolean lattice; de Bruijn's theorem corresponds to the case when the subposet in question consists of a single ground element.  This result is equivalent to showing that de Bruijn cycles exist for the possible ways of filling in the elements of a collection of sets, labeled $A,B,C$, etc., whose Venn Diagram has a shape with specified intersection patterns.

The landmark work of \cite{cdg} extended the agenda of finding de Bruijn cycles to possibly the next most natural set of combinatorial objects, namely $k$-subsets of $[n]$.  Here the authors conjectured that de Bruijn cycles for the $k$-subsets of $[n]$ existed if and only if $n$ was sufficiently large and $n\big|{n\choose k}$, which is an obvious necessary condition for the existence of the cycle in the natural coding (where, e.g., the set $\{a,b,c\}$ is encoded in the string by some permutation of its elements) since in this case each element of the set must appear equally often.  In this area, important contributions have been those of \cite{h} and \cite{r}.  In \cite{h}, Hurlbert makes significant progress for the case of $k=3,4,6$; the case of $k=4$ is reduced to verification of two ground cases in \cite{r}.  Here we follow the direction of \cite{bg}, who proved that, in a suitable encoding, de Bruijn cycles can be created for the subsets of $[n]$ of size in the interval $[s,t]; 0\le s<t\le n$.  For example, the string 1110011010 is a de Bruijn cycle of all the 2- and 3-subsets of [4] if one uses a sliding window of length 4 and the characteristic vector coding for subsets.  Once again, we remind the reader that setting $s=0$ and $t=n$ and using a full (length $n$) encoding window for a subset yields de Bruijn's theorem!  Thus the result in \cite{bg} can be viewed as a restricted version of de Bruijn's theorem.  In this paper we generalize this result to exhibit de Bruijn cycles for words of length $n$ over a $k$-ary alphabet and having weight between $s$ and $t$, where the parameters $s$ and $t$ are suitably restricted.  Viewed differently, this proves existence of de Bruijn cycles for multisets having between $s$ and $t$ elements, and where each element of $[n]$ may appear at most $k-1$ times in the multiset.  Setting $k=2$ enables us to retrieve the result in \cite{bg} on subsets of cardinality between $s$ and $t$.  Results for de Bruijn cycles of multisets using the ``classical" coding may be found in \cite{hjz}.

A good survey of some recent and not-so-recent papers on de Bruijn cycles may be found in \cite{hh}.  See also \cite{bks}, where de Bruijn cycles are proved to exist for graphs, hypergraphs, and such -- using the simple and far-reaching device of invoking de Bruijn's theorem, where it all began, together with using a suitable encoding.  And, last but not least, the reader is urged to study the very recent \cite{dg} and order it for his/her library.

\section{Labeled Subposets of the Boolean Lattice, a.k.a.~Venn Diagrams with Specified Intersection Patterns} A poset $(P,\le)$ is a set $P$ with a relation $\le$ on $P$ that is reflexive, transitive, and antisymmetric.  The Hasse diagram of a poset is a representation of $(P,\le)$ using a pictorial device that represents its {\it transitive reduction}.  Specifically, an ``upward edge" is drawn from $x$ to $y$ if $x<y$ and there is no $z$ such that $x<z<y$. The main result of this section is the following:
\begin{thm} Let $(P,\le)$ be a fixed poset with elements of $P$ being subsets of $[n]$ and $A\le B$ if $A\subseteq B$.  Assume furthermore that the Hasse diagram of $P$ is consistent with that of the Boolean Lattice $\cb_n$.  There are then $\alpha^n$ assignments of the elements of $[n]$ to the elements of $P$, where $\alpha$ is the number of antichains in $P$.  Furthermore, there exists a de Bruijn cycle of these assignments.\end{thm}
\begin{proof} We start with the first part of the result.  Given any element $j\in[n]$ it is clear that if $j\in A$ for some $A\in P$, we must have $j\in B$ for each $B\ge A$, since $\le$ is the inclusion relation.  It thus follows that the number of ways to assign element $j$ to the sets in $P$ is equal to the number of ways of labeling the elements of $P$ with zeros and ones so that if $A$ is labeled by a 1, then so is each $B\supseteq A$; this includes the case where all the elements of $P$ are labeled with zeros.
We call a coloring $c:P\rightarrow\{0,1\}$ {\it unitarily up-closed} if $c(A)=1\Rightarrow c(B)=1$ for each $B\in P$ with $B\ge A$.  We will prove below that the set of unitarily up-closed colorings is in bijection with the antichains of $(P,\le)$; note that $\emptyset$ is considered to be an antichain in $P$, as are all one-element sets:

For one direction, let $c$ be a unitarily up-closed coloring on $(P,\le)$. Let $M_c := \{A\in P : c(A) = 1\ {\rm and\ }
c(B) = 0\ {\rm for\ every\ } B\le A\}$ i.e., the set of elements of $P$ minimal with respect to receiving the value 1 under
$c$. Clearly, $M_c$ forms an antichain, since if there exist $A,B\in M_c$ with $A\le B$, then $A$ does not meet the constructive criteria
of $M_c$. 
This provides a map from colorings to antichains. To see that the map is injective, let $c$ and $d$ be two
different unitarily up-closed colorings, and suppose that $M_c = M_d$. Fix $A\in P$. Suppose that $c(A) = 1$. If $A$ is
minimal, then we have that $d(A)$ = 1 since $M_c = M_d$. If $A$ is not minimal, it is above some $B\in M_c$. $B\in M_d$
by hypothesis, and since $B\le A$, $d(A) = 1$ as $d$ is unitarily up-closed.   Now suppose that $c(A) = 0$. Since $A$
is not above {\it any} element of $M_c$, $A$ is also not above any element of $M_d$. Hence $d(A) = 0$, as if $d(A) = 1$ we
would have that $A$ was above some minimal element under $d$.

To see the reverse, notice that for a fixed antichain $\ca$, we can define a unitarily closed up-coloring $c=c_\ca$
by the rule $c(A) = 1$ iff $B\le A$ for some $B\in\ca$, 0 otherwise. We claim that
$M_c = \ca$, and hence this provides a $1-1$ inverse to our $1-1$ function. Since no element below any element
of $\ca$ receives color 1, clearly $\ca\subseteq M_c$. On the other hand, every element $A$ with $c(A)=1$ is above some
element of $\ca$, and since $M_c$ contains every element of $\ca$, no other element receiving color 1 can belong to
$M_c$, and we are done.

It follows that the ``fate" of each element of $[n]$, i.e., which sets in $P$ it belongs to, can be determined in $\alpha$ ways, and thus there are $\alpha^n$ assignments of the elements of $[n]$ to the sets in $P$.

To prove the second part, we note that each of the $\alpha$ ways of assigning element $j$ to the sets in $P$ can be coded using letters from an alphabet $\Lambda$ of size $\alpha$.  By de Bruijn's theorem, there exists a cycle of the elements of $\Lambda^{[n]}$, and thus any such de Bruijn cycle may be viewed as a de Bruijn cycle of the assignment of elements of $[n]$ to the sets in $P$.  This completes the proof; specific examples follow.
\hfill\end{proof}

\medskip
\noindent {\bf Examples.}  The two letter words on a three letter alphabet $\{a,b,c\}$ may be cycled as follows:
\[cc \rightarrow ca \rightarrow aa \rightarrow ab \rightarrow bb \rightarrow bc \rightarrow cb \rightarrow ba \rightarrow ac.\]
We let the letters of the alphabet code for the ``allowable configurations" corresponding to the three antichains of the 2-chain with Hasse diagram
\begin{center}
\begin{tikzpicture}
\node (max) at (0,4) {$B$};
\node (b) at (0,3) {$A$};
\draw (max) -- (b);
\end{tikzpicture}
\end{center}
This is done by using a coding such as 
\[a={{0}\atop{0}}; \quad b={{1}\atop{0}}; \quad c={{1}\atop{1}},\]
where, e.g., ${{1}\atop{0}}$ indicates that $j\in B; j\not\in A$.
Thus the de Bruijn cycle $ccaabbcba$ above may be written as 
\[{{1}\atop{1}}\quad {{1}\atop{1}}\quad {{0}\atop{0}}\quad {{0}\atop{0}}\quad {{1}\atop{0}}\quad {{1}\atop{0}}\quad {{1}\atop{1}}\quad {{1}\atop{0}}\quad {{0}\atop{0}}\]
To decode the cycle, we read the above stack from left to right in sliding groups of two columns at a time, the first and second columns describing to which sets elements $1$ and $2$ belong.  Since $n=2$, this is all we need.
The actual de Bruijn cycle of assignments of $\{1,2\}$ to the two sets in the poset is then as follows:
\begin{center}
\begin{tikzpicture}
\node (max) at (0,4) {$ \{1,2\} $};
\node (b) at (0,3) {$ \{1,2\} $};
\draw (max) -- (b);
\end{tikzpicture} \begin{tikzpicture}
\node (max) at (0,4) {$ \{1\} $};
\node (b) at (0,3) {$ \{1\} $};
\draw (max) -- (b);
\end{tikzpicture} \begin{tikzpicture}
\node (max) at (0,4) {$ {\emptyset } $};
\node (b) at (0,3) {$ {\emptyset } $};
\draw (max) -- (b);
\end{tikzpicture} \begin{tikzpicture}
\node (max) at (0,4) {$ \{2\} $};
\node (b) at (0,3) {$ {\emptyset} $};
\draw (max) -- (b);
\end{tikzpicture} \begin{tikzpicture}
\node (max) at (0,4) {$ \{1,2\} $};
\node (b) at (0,3) {$ {\emptyset } $};
\draw (max) -- (b);
\end{tikzpicture} \begin{tikzpicture}
\node (max) at (0,4) {$ \{1,2\} $};
\node (b) at (0,3) {$ \{2\} $};
\draw (max) -- (b);
\end{tikzpicture} \begin{tikzpicture}
\node (max) at (0,4) {$ \{1,2\} $};
\node (b) at (0,3) {$ \{1\} $};
\draw (max) -- (b);
\end{tikzpicture} \begin{tikzpicture}
\node (max) at (0,4) {$ \{1\} $};
\node (b) at (0,3) {$ {\emptyset } $};
\draw (max) -- (b);
\end{tikzpicture} \begin{tikzpicture}
\node (max) at (0,4) {$ \{2\} $};
\node (b) at (0,3) {$ \{2\} $};
\draw (max) -- (b);
\end{tikzpicture}
\end{center}
The same may be done for a poset with any Hasse diagram.  If, e.g., we have the W-poset with 13 antichains, then 

(i) We start with a de Bruijn cycle for all the $n$-letter words on an alphabet such as $\{A,B,\ldots, M\}$.  

(ii) We then rewrite this as an assignment of the elements of $\{1,2\ldots,n\}$ to the five sets in the poset, using a sliding stack of height 5 and width $n$.  

(iii) Finally, we may ``draw" the $13^n$ realizations of the poset, i.e., the assignments of $n$ elements to 13 sets in a  Venn diagram.  

\noindent Note, however, that going from steps (i) to (ii) to (iii) leads to a sequential loss of data compression, and thus (i) gives the recommended de Bruijn cycle.  
\section{Words with Weights in Prescribed Intervals}  The main result of this section is the following:
\begin{thm}  Let $k,n\in{\mathbb Z}^+$.  Consider $n$ letter words $w=(w_1,w_2,\ldots,w_n)$ on the $k$-letter alphabet $\Lambda=\{0,1,\ldots,k-1\}$, and define the weight $h(w)$ of $w$ by $h(w)=\sum_{i=1}^nw_i$.  Let $s,t\in{\mathbb Z}^+$ satisfy $0\le s<s+k-1\le t\le n(k-1)$. Let $\cw$ be the collection of all such words with weights between $s$ and $t$. Then there exists a de Bruijn cycle of the elements of $\cw$.
\end{thm}
\begin{proof}  We create a digraph $D=(V,E)$ with vertex set equal to the set of all $n-1$-letter words over the alphabet $\Lambda$ and having weights between $s-(k-1)$ and $t$.  A directed edge is drawn from vertex $v_1$ to vertex $v_2$ if (i) the last $n-2$ letters of $v_1$ are the same as the first $n-2$ letters of $v_2$ and if (ii) the edge labeled by the concatenation of the letters in $v_1$ and $v_2$ yields an $n$-letter word with weight between $s$ and $t$.

It is clear that the indegree $i(v)$ and outdegree $o(v)$ of any vertex $v$ are both equal to the number of prefixes (or suffixes) that yield an edge label of legal weight; if, e.g., $h(v)=s-(k-1)$ or $h(v)=t$, then $i(v)=o(v)=1$.  We will next show that $D$ has a single weakly connected component; this will show that it is Eulerian\cite{west}, and the Eulerian path will give the required de Bruijn cycle.  

Weak connectedness will established by showing that there exists a path between any vertex in $D$ and a special sink vertex $SV$ of weight $s$ that consists of a uniquely determined number $a\ge1$ of letters $x$ followed by $n-1-a\ge0$ letters that are all $x+1$.  The rest of the proof of Theorem 2 consists of  demonstrating four embedded lemmas, that state that we can (i) connect any vertex of weight $\ge s+1$ to one of weight $s$; (ii) connect any vertex of weight $\le s-1$ to one of weight $s$; (iii) connect any vertex of weight $s$ to one with weight $s$ and letters $x, x+1$; and (iv) connect any vertex with the right number of $x$s and $x+1$s to $SV$.
\begin{lem}  Let $v$ have weight $h(v)$ in the range $[s+1,t]$.  Then there is a path from $v$ to another vertex of weight $s$.
\end{lem}
\begin{proof} We first cycle any zeros in the front of the vertex to the end.  These steps may be taken without changing the weight of the vertex, and so that the weight of the words representing the edges is also maintained at $h(v)$.  The resulting vertex, $v'$, has weight $h(v')=h(v)$ that may be as high as $t$, and has first letter different from zero.  If $h(v')=t$, then the only allowable letter we can add, when the first letter is dropped, is zero.  This maintains the edge weight at $t$, while leading to the new vertex weight $h(v'')$ satisfying the conditions $t>h(v'')\ge t-(k-1)\ge s$. If $h(v')=t-r\ge s+1$, we replace the first letter $f$ by $\min\{f-1, r\}$.  This either reduces the vertex weight by one or changes it to $t-r-f+r=t-f\ge t-(k-1)\ge s$.  The edge weight might increase from $t-r$ to $t$, but this is OK.  We now see that $v''$ either has weight $s$ or, if not, has weight smaller than $h(v')$.  We repeat the above process until we reach a vertex with weight $s$. \hfill\end{proof}
\begin{lem}  Let $v$ have weight $h(v)$ in the range $[s-(k-1),s-1]$.  Then there is a path from $v$ to another vertex of weight $s$.
\end{lem}
\begin{proof}If $h(v)=s-(k-1)$, we drop the first letter and add the letter $k-1$.  This makes the edge label equal to $s$, and the vertex weight does not decrease.  If the vertex label has increased, we have made progress towards our target.  If it has not, due to the first letter in $v$ being $k-1$, we repeat the process till we get to a vertex with first letter smaller than $k-1$.  If the starting weight is $h(v)=s-r; 1\le r\le k-2$, we drop all letters $k-1$ at the front of the word and append $k-1$s to the end until the first letter is smaller than $k-1$, thus arriving at the vertex $v'$.    We now add the letter $\max\{r, f+1\}$, where $f$ is the first letter of the vertex $v'$.  This leads to the edge weight becoming either $s$ (if $r\ge f+1$) or $s-r+f+1\le t$ (if $r\le f+1$).  The vertex weight goes up by 1, if $f+1\ge r$ and becomes $s-f$ if $f+1\le r$.  We now iterate the above process until the vertex weight becomes $s$.\hfill
\end{proof}
\begin{lem}  Let $v=(v_1,\ldots,v_{n-1})$ have weight $h(v)=s$.  Then there is a path from $v$ to another vertex of weight $s$ and consisting entirely of letters that are either $x$ or $x+1$ (and at least one letter $x$).
\end{lem}
\begin{proof}  Type $A$ letters are those in the set $\{0,1,\ldots,x\}$, while type $B$ letters belong to the set $\{x+1,\ldots,k-1\}$.  Let $a$ and $b$ be the desired number of $x$s and $x+1$s.  Depending on how many type $A$ and type $B$ letters $v$ has, we will need the path from $v$ to decrease a certain number of type $B$ letters to either $x$ or $x+1$, and to increase some or all of the type $A$ letters to either $x$ or $x+1$.  Start with $v_1$.  If $v_1$ is a type $A$ letter we change it to either $x$ or $x+1$ as needed.  This possibly increases the vertex and edge weights, and the condition $t\ge s+(k-1)$ keeps both weights legal.  (The weights don't have to increase since we may replace the $x$ with an $x$, causing no change in vertex weight; or $x$ might equal 0, in which case there is no change in edge weight.)  If $v_1$ is of type $B$ we again change it to either $x$ or $x+1$, causing a possible drop in vertex weight.  We now repeat the process with the first letter $v_2$ of the new vertex $v'$ {\it as long as the new vertex weight $h(v')$ satisfies}   $t-h(v')\ge x$ or $h(v')-(s-(k-1))\ge x$.  No edge traversal in the digraph is undertaken that leads to a vertex that is closer than $x$ in weight from the two extreme vertex weights, namely $s-(k-1)$ and $t$.  If such a ``dangerous"  occurrence is imminent, we abort such a move, cycling instead until we have the opportunity to increase a ``dangerously low weight" or decrease a ``dangerously high vertex weight."  The process is repeated until we have the required numbers $a,b$ of symbols $x$ and $x+1$ respectively.  An example is given after the proof of Corollary 3.
\hfill\end{proof}
\begin{lem}  Let $v$ have weight $h(v)=s$ and be composed entirely of $x$s and $x+1$s as in Lemma 3.  Then there is a path from $v$ to the sink vertex $SV=(x,x,\ldots,x,x+1,\ldots,x+1)$ having the same number of $x$s and $x+1$s as $v$.
\end{lem}
\begin{proof}  The proof is similar to that of Lemma 5.  We first identify both $x$s and $x+1$s in $v$ that are ``out of place".  Clearly the number of out of place $x$s must equal the number of out of place $x+1$s, and the latter all appear before the former.  We cycle the letters of the word until we arrive at the first out of place $x+1$, delete it, and add an $x$ at the end.  This decreases the vertex weight by one.  We continue cycling the word and replacing $x+1$s by $x$s until the vertex weight is within $x$ of the minimum legal weight, i.e. $s-(k-1)$.  The next phase is to increase the vertex weight by replacing (cyclically) out of place $x$s by $x+1$s, until the vertex weight is within $x$ of the maximum allowable, i.e. $t$.  We  alternate this process until the two letters are all cyclically in the right places, and then cycle until we reach $SV$.
\end{proof}
\noindent Lemmas 3, 4, 5, and 6 together complete the proof of Theorem 2.\hfill\end{proof}
We define a {\it Redundant de Bruijn Cycle} of a collection of $\ca$ combinatorial objects  to be  a cycle of length $\ca'>\ca$ so that each object appears exactly once as a set of consecutive elements in the cycle, and thus so that $\ca'-\ca$ consecutive elements are redundant objects of another type.
\begin{cor}  For each $t\ge k-1$, there exists a redundant de Bruijn cycle of the $\ca(n,k,t)$ $n$-letter weight $t$ words over $\{0,1,\ldots,k-1\}$ that is of length $\ca'(n,k,t)=\ca(n,k,t)(1+o(1)); k,t$ fixed, $n\to\infty$.
\end{cor}
\begin{proof}  We set $s=t-(k-1)\ge0$ in Theorem 2, and obtain a de Bruijn cycle of all words with weight between $t-(k-1)$ and $t$.  The length of this cycle is
\[\sum_{j=t-(k-1)}^t \ca(n,k,j)=\ca(n,k,t)(1+o(1)),\]
since for fixed $t$ and large $n$, $\ca(n,k,j)$ is a monotone function of $j$; $0\le j\le t$ with $\ca(n,k,j)/\ca(n,k,j+1)\to0,\ n\to\infty$.  This completes the proof.
\hfill\end{proof}

\medskip

\noindent {\bf An Example:  The Algorithm in Lemma 5 in Action.}  We choose $v$ to be $(0,0,0,2,2,5,5,5,3,3)$; $s=25, t=30, k=6, n=11$. The target vertex is $SV=(2,2,2,2,2,3,3,3,3,3)$.  Legal vertex weights are between 20 and 30, and legal edge weights are between 25 and 30.  We declare a dangerous situation to be one in which further progress would lead to vertex weights of 20, 21, 29, or 30 -- unless we cycle to a vertex with appropriate first letter.  Below, red numbers represent vertex weights, blue numbers represent edge weights, and the symbol D represents ``Danger."  We proceed as follows:
\begin{center}
{0,0,0,2,2,5,5,5,3,3} {\color{red} 25}\\
$\downarrow$ {\color{blue} 28} \\
{0,0,2,2,5,5,5,3,3,3} {\color{red} 28, D}\\
$\downarrow$ {\color{blue} 28} \\
{0,2,2,5,5,5,3,3,3,0} {\color{red} 28, D}\\
$\downarrow$ {\color{blue} 28} \\
{2,2,5,5,5,3,3,3,0,0} {\color{red} 28, D}\\
$\downarrow$ {\color{blue} 30} \\
{2,5,5,5,3,3,3,0,0,2} {\color{red} 28, D}\\
$\downarrow$ {\color{blue} 30} \\
{5,5,5,3,3,3,0,0,2,2} {\color{red} 28}\\
$\downarrow$ {\color{blue} 30} \\
{5,5,3,3,3,0,0,2,2,2} {\color{red} 25}\\
$\downarrow$ {\color{blue} 27} \\
{5,3,3,3,0,0,2,2,2,2} {\color{red} 22, D}\\
$\downarrow$ {\color{blue} 27} \\
{3,3,3,0,0,2,2,2,2,5} {\color{red} 22, D}\\
$\downarrow$ {\color{blue} 25} \\
{3,3,0,0,2,2,2,2,5,3} {\color{red} 22, D}\\
$\downarrow$ {\color{blue} 25} \\
{3,0,0,2,2,2,2,5,3,3} {\color{red} 22, D}\\
$\downarrow$ {\color{blue} 25} \\
{0,0,2,2,2,2,5,3,3,3} {\color{red} 22}\\
$\downarrow$ {\color{blue} 25} \\
{0,2,2,2,2,5,3,3,3,3} {\color{red} 25}\\
$\downarrow$ {\color{blue} 28} \\
{2,2,2,2,5,3,3,3,3,3} {\color{red} 28, D}\\
$\downarrow$ {\color{blue} 30} \\
{2,2,2,5,3,3,3,3,3,2} {\color{red} 28, D}\\
$\downarrow$ {\color{blue} 30} \\
{2,2,5,3,3,3,3,3,2,2} {\color{red} 28, D}\\
$\downarrow$ {\color{blue} 30} \\
{2,5,3,3,3,3,3,2,2,2} {\color{red} 28, D}\\
$\downarrow$ {\color{blue} 30} \\
{5,3,3,3,3,3,2,2,2,2} {\color{red} 28}\\
$\downarrow$ {\color{blue} 30} \\
{3,3,3,3,3,2,2,2,2,2} {\color{red} 25}\\
$\downarrow$ {\color{blue} 28} \\
{3,3,3,3,2,2,2,2,2,3} {\color{red} 25}\\
$\downarrow$ {\color{blue} 28} \\
{3,3,3,2,2,2,2,2,3,3} {\color{red} 25}\\
$\downarrow$ {\color{blue} 28} \\
{3,3,2,2,2,2,2,3,3,3} {\color{red} 25}\\
$\downarrow$ {\color{blue} 28} \\
{3,2,2,2,2,2,3,3,3,3} {\color{red} 25}\\
$\downarrow$ {\color{blue} 28} \\
{2,2,2,2,2,3,3,3,3,3} {\color{red} 25}
\end{center}
thus arriving at the sink vertex.

\section{Further Research}  It would be interesting to improve Theorem 2 so that the range of the weight of the word is as small as possible, or else to prove that the range in Theorem 2 is the best possible.   Also, one might ask how and to what extent one can show existence of de Bruijn cycles for unordered posets.  Last but not least, can results be proved for (labelled as well as unlabelled) subposets of mother posets other than the Boolean Lattice?

\section{Acknowledgments} The research of AG was supported by NSF Grant 1004624.  

\end{document}